\documentclass[11pt, reqno, a4paper]{amsart}

\usepackage{amssymb, amsmath, amscd}
\usepackage{amsthm}

\usepackage{verbatim}

\usepackage{color}

\usepackage{etex}
\usepackage{dsfont}
\usepackage[matrix, arrow, curve, frame, arc]{xy}
\usepackage{pgf}
\usepackage{tikz-cd}
\usetikzlibrary{arrows,shapes,automata,backgrounds,petri,calc}
\usepackage[square,sort,comma,numbers]{natbib}
 \usepackage{url}
 \usepackage{hyperref}
 \usepackage[shortlabels]{enumitem} %to index the enumeration

\newcommand{\tikzAngleOfLine}{\tikz@AngleOfLine}
\def\tikz@AngleOfLine(#1)(#2)#3{%
\pgfmathanglebetweenpoints{%
\pgfpointanchor{#1}{center}}{%
\pgfpointanchor{#2}{center}}
\pgfmathsetmacro{#3}{\pgfmathresult}%
}

\newcommand{\bN}{\mathbb{N}}

\newcommand{\ba}{\bar{\alpha}}
\newcommand{\La}{\Lambda}
\newcommand{\vf}{\varphi}

\newcommand{\ul}{\underline}
\newcommand{\rad}{\operatorname{rad}}
\renewcommand{\mod}{\operatorname{mod}}
\newcommand{\Hom}{\operatorname{Hom}}

\begin{document}

\newtheorem{defi}{Definition}[section]
\newtheorem{rem}[defi]{Remark}
\newtheorem{prop}[defi]{Proposition}
\newtheorem{ques}[defi]{Question}
\newtheorem{lemma}[defi]{Lemma}
\newtheorem{cor}[defi]{Corollary}
\newtheorem{thm}[defi]{Theorem}
\newtheorem{expl}[defi]{Example}

\baselineskip=16pt
\parindent0pt

\title[Period four]{Tame symmetric algebras of period four}

\author[K. Erdmann]{Karin Erdmann}
\address[Karin Erdmann]{Mathematical Institute, University of Oxford, UK}
\email{erdmann@maths.ox.ac.uk} 

\author[A. Hajduk]{Adam Hajduk}
\address[Adam Hajduk]{Faculty of Mathematics and Computer Science, 
Nicolaus Copernicus University, Chopina 12/18, 87-100 Torun, Poland}
\email{ahajduk@mat.umk.pl}

\author[A. Skowyrski]{Adam Skowyrski}
\address[Adam Skowyrski]{Faculty of Mathematics and Computer Science, 
Nicolaus Copernicus University, Chopina 12/18, 87-100 Torun, Poland}
\email{skowyr@mat.umk.pl} 


\begin{abstract} In this paper we are concerned with the structure of tame symmetric algebras $A$ of period four (TSP4 algebras, 
for short). We will mostly focus on the case when the Gabriel quiver of $A$ is biserial, i.e. there are at most two arrows ending and 
at most two arrows starting at each vertex, but some of the results can be easily extended to general case. Here, we serve a basis for 
upcoming series of articles devoted to solve the problem of classification of all TSP4 algebras with biserial Gabriel quiver. We present 
a range of properties (with relatively short proofs) which must hold for the Gabriel quiver of a tame symmetric algebra of period four. 
Amongst others we show that triangles (and squares) appear naturally in the Gabriel quivers of such algebras, so as for weighted surface 
algebras \citep{WSA,WSA-GV,WSA-corr}. \end{abstract}

\maketitle

\section{Introduction}\label{sec:1} 

Classical examples of tame symmetric algebras of period four are 2-blocks of finite-dimensional group algebras with quaternion defect 
groups. More recently it was discovered that all weighted surface algebras \cite{WSA} (see also \cite{WSA-GV} and \cite{WSA-corr}) are 
tame symmetric of period four, and so are virtual mutations investigated in \cite{HSS} or so called weighted generalized triangulation 
algebras \cite{SS}, which generalize both mentioned classes. The main result of \cite{AGQT} established the classification of tame 
symmetric algebras of period four whose Gabriel quiver is 2-regular, which gives an evidence that general classification may be in 
reach after some work. A full classification in the biserial case seems to be an exciting challenge. \medskip 

This paper is a contribution towards this goal. Here we present a range of properties with short proofs, but which will be essential 
input for the general classification (work in progress). \medskip 

Throughout we fix an algebraically closed field, and we consider finite-dimensional associative $K$-algebras with identity. We also 
assume that algebras are basic and connected. Recall that an algebra $A$ is {\it self-injective}, provided that $\La$ is injective 
as a right $\La$-module, i.e. projective modules are also injective (see also \cite{Sk}). In this paper, we focus our attention on 
{\it symmetric} algebras, that is these self-injective algebras, for which there is a nondegenerate symmetric $K$-bilinear form 
$\La\times \La\to K$. There are many classical examples of symmetric algebras, for instance, blocks of finite-dimensional group 
algebras \cite{E1} or Hecke algebras associated to Coxeter groups \cite{ad1}. Any algebra $\La$ is a quotient of its trivial extension 
$T(\La)$, which is a symmetric algebra. \medskip
 
For an algebra $\La$ we denote by $\mod \La$ the category of finitely generated (right) $\La$-modules. For a module $M$ in $\mod \La$, 
its {\it syzygy} is a module $\Omega(M)=ker(\pi)$, where $\pi:P\to M$ is a projective cover of $M$ in $\mod\La$ (so syzygy is defined 
up to isomorphism).\medskip 

We call a module $M$ in $\mod \La$ a {\it periodic module} if $\Omega^d(M)\cong M$, for some $d\geqslant 1$ (the smallest such $d$ is 
the {\it period} of $M$). Recall that an algebra $\La$ is called a {\it periodic algebra} if $\La$ is periodic as an $\La$-bimodule, 
or equivalently, $\La$ is a periodic module over its enveloping algebra $\La^e=\La\otimes_K\La$. Periodicity of an algebra implies 
periodicity of all non-projective indecomposable $A$-modules (see for example \cite[Theorem IV.11.19]{SkY}). In particular, if $\La$ is 
a periodic algebra, then all simple $\La$-modules are periodic. Moreover, it is known \cite[see Theorem 1.4]{GSS} that periodicity of 
simples in $\mod\La$ implies $\La$ is self-injective, and hence, periodic algebras form a subclass in the class of self-injective algebras. 
\medskip 

Here we work with bound quiver algebras $\La=KQ/I$, where the Gabriel quiver $Q$ is biserial: that is, at each vertex at most two 
arrows start and at most two arrows end. We will consider algebras $\La$ which are both symmetric and tame, and we assume that $\La$ 
is a periodic algebra of period four. Any such algebra is said to be a TSP4 algebra. We will give an overview of general properties of 
Gabriel quivers $Q$ and minimal generators of ideals $I$ for such algebras. A full classification by quivers and relations requires much 
more efforts. In particular, we shall see that triangles (and squares) appear naturally; see Section \ref{sec:4}. Moreover, in the last 
section, we present partial results describing some distinguished types of vertices. \medskip 

For the necessary background in the representation theory we refer to books \cite{ASS, SkY}. \bigskip 

\section{Preliminaries}\label{sec:2}  

Let $\La = KQ/I$ be an admissible presentation of $\La$, where the algebra is tame and symmetric, and has $\Omega$-period $4$, as an 
algebra. In particular,  all simple modules are $\Omega$-periodic as $A$-modules with period dividing $4$ \cite[Theorem IV.11.19]{SkY}. 
In fact, we can assume that all simples have period $4$ (see Remark \ref{period2}). We also assume $Q$ is connected, that is $\La$ is 
indecomposable as an algebra. For a vertex $i\in Q$, we denote by $P_i$ the indecomposable projective module in $\mod\La$ associated 
to vertex $i$, and by $p_i$ its dimension vector $p_i:=\ul{\dim}(P_i)$. Similarly, we write $S_i$ and $s_i$, for the simple module 
associated to vertex $i$ and its dimension vector. \smallskip 

For a vertex $i$ of the quiver $Q$, we let $i^-$ be the set of arrows ending at $i$, and $i^+$ the set of arrows starting at $i$.
In this paper, we assume  the sizes $|i^-|$ and $i^+|$ are  at most $2$. With this, $Q$ is said to be {\it 2-regular} if 
$|i^-|=|i^+|=2$, and {\it biserial} if $1\leq |i^-|, |i^+|\leq 2$. We say that $i\in Q_0$ is a {\it regular} vertex ($1$- or $2$-regular), 
provided $|i^-|=|i^+|$ (and the size is equal $1$ or $2$, respectively). Otherwise, we call $i$ a {\it non-regular} vertex. \smallskip 

We will use the following notation and convention for arrows: we write $\alpha, \ba$ for the arrows starting at vertex $i$, with the 
convention that $\ba$ does not exist in case $|i^+|=1$. Similarly we write $\gamma, \gamma^*$ for the arrows ending at some vertex $i$,
where again $\gamma^*$ may not exist. \smallskip 

Then $Q$ has a subquiver 
$$
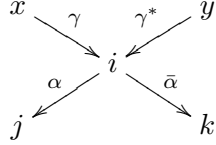
\xymatrix@R=0.3cm{x\ar[rd]^{\gamma}&&y\ar[ld]_{\gamma^*}\\&i\ar[ld]_{\alpha}\ar[rd]^{\ba}&\\j&&k}$$ 

Consider the simple module $S_i$, $i\in Q_0$. We will briefly discuss some basic consequences of $\Omega$-periodicity of $S_i$, 
mainly, the associated exact sequence. Recall that there are natural isomorphisms $\Omega(S_i)=\rad P_i=\alpha\La+\ba\La$ and 
$\Omega^-(S_i)\cong (\gamma,\gamma^*)\Lambda\subset P_x\oplus P_y$. In particular, it follows that the module $P_i^+=P_j\oplus P_k$ 
is a projective cover of $\Omega(S_i)$ and the module $P_i^-=P_x\oplus P_y$ is an injective envelope of $\Omega^-(S_i)$ ($\La$ is 
symmetric). Consequently, involving $\Omega$-periodicity (period $4$) of $S_i$, we conclude that there is an exact sequence in $\mod\La$ 
of the form 
$$0\to S_i\to P_i \stackrel{d_3}\to  P_i^- \stackrel{d_2}\to  P_i^+ \stackrel{d_1}\to P_i \to S_i\to 0 \leqno{(*)}$$ 
with $Im(d_k)\cong\Omega^k(S_i)$, for $k\in\{1,2,3\}$. By our convention, $P_y$ or $P_k$ may not exist. Moreover, we denote by $p_i^+$ 
(respectively, $p_i^+$) the dimension vector $\ul{\dim}(P_i^+)$ (respectively, $\ul{\dim}(P_i^-)$). Using the above sequence, one 
easily gets that $p_i^+=p_i^-$. We use this fact (without mentioning) many times in the rest part of the paper. \smallskip 

Now, we will show a few examples of results obtained by using exact sequences of the form $(*)$. As a first application note the 
following lemma. 

\begin{lemma}\label{lem:2.1} 
If $\La$ has infinite type then there is no arrow $\alpha: i\to j$ with $i^+ = \{ \alpha\} = j^-$. \end{lemma}

\begin{proof} Suppose there is such an arrow. Then $\Omega(S_i) = \alpha\La \cong \Omega^{-1}(S_j)$ and $\Omega^2(S_i)\cong S_j$.
Therefore in the exact sequence for $S_i$ the projective $P_i^-$ is isomorphic to $P_j$ and this means that there is a unique arrow 
ending at $i$ and it starts at $j$.

As well, in the exact sequence for $S_j$ we have $P_j^+ \cong P_i$ since $\Omega^2(S_j)\cong S_i$. Therefore there is a unique arrow 
starting at $j$ and it ends at $i$. Now, $Q$ is connected and hence has only two vertices and two arrows. Then $\La$ is a Nakayama 
algebra of finite representation type, hence a contradiction (see for example \cite[Theorems I.10.3 and 10.7]{SkY}). \end{proof} 

\begin{rem}\label{period2} \normalfont Actually, existence of arrow with the above described property implies that $\La$ is of finite 
type, as it is explained in the following note \cite{E2}. There is also proved this condition is equivalent to existence of a simple 
with period $2$. Hence, when dealing with TSP4 algebras of infinite type, we may assume that all simples have period exactly $4$. 
\end{rem} 

We have also the following observation. 

\begin{lemma}\label{lem:2.2} The quiver $Q$ does not have a subquiver of the form
$$j  \stackrel{\longrightarrow}\longleftarrow  i \longleftarrow t$$
where all arrows to and from $i$ are shown. \end{lemma}

\begin{proof} Assume this happens. Then in the exact sequence for $S_i$ we have $P_i^+ \cong P_j$ and $P_i^-\cong P_j \oplus P_t$. 
Since $P_t\neq 0$ it follows that $\ul{\dim} P_i^+  \neq  \ul{\dim} P_i^-$, a contradiction. \end{proof} \medskip

To end this preliminary section we will give one simple lemma pertaining vectors $p_i^+=p_i^-$, for $i\in Q_0$ (this common dimension 
vector of two modules $P_i^+$ and $P_i^-$ will be denoted by $\hat{p}_i$). \smallskip 

It is clear from the exact sequence $(*)$ that $p_i$ is less or equal to $\hat{p}_i+s_i$ (in the product order), since 
$p_i-s_i=\ul{\dim}\Omega^1(S_i)$ is less than $\ul{\dim}(R^+)=\hat{p}_i$. Moreover, $\hat{p}_i$ is greater up to dimension, as 
the following shows (here we write $|x|$ for the sum $|x|=x_1+\dots+x_n$, where $x=(x_1,\dots,x_n)\in\bN^n$, which corresponds to the 
$K$-dimension of $X$, if $x=\ul{\dim}(X)$, for a module $X$ in $\mod\La$). 

\begin{lemma}\label{lem:1} $|\hat{p}_i|>|p_i|$. \end{lemma} 

\begin{proof} Of course, we have an exact sequence $0\to \Omega^2(S_i)\to P_i^+ \to \Omega^1(S_i)\to 0$, where $\Omega^1(S_i)=\rad P_i$ 
has dimension vector equal to $p_i-s_i$. We claim that 
$$(\Box)\qquad\qquad |\hat{p}_i|-\dim_K\Omega^1(S_i)>1.$$ 
Indeed, if this is not the case, then the difference is $1$, and we conclude that $\Omega^2(S_i)\cong S_i$. On the other hand, $P_i^+$ 
(respectively, $P_i^-$) are injective envelope (respectively, projective cover) of $\Omega^2(S_i)$, so it would imply that both 
are isomorphic to $P_i$. It means that there is a unique arrow in $Q$ starting at $i$ which also ends at $i$, and dually, there 
is a unique arrow in $Q$ ending at $i$ which also starts at $i$. As a result $Q$ admits one vertex and two loops, which is impossible, 
due to our assumptions. Therefore, $(\Box)$ holds. In particular, we get $|\hat{p}_i|-\dim_K\Omega^1(S_i)=|\hat{p}_i|-|p_i|+1>1$, and 
hence $|\hat{p}_i|-|p_i|>0$, so we are done. \end{proof} 

\section{Period 4 and minimal relations}\label{sec:3}

In this section, we develop further consequences of the structure of the exact sequence $(*)$ associated to the simple module 
$S_i$, as described in the previous section. Actually, we will focus rather on maps and show their connection with minimal relations 
defining algebra $\La$. \smallskip 

We start with our given presentation $\La=KQ/I$ and a vertex $i\in Q_0$. We will briefly write $J$ for the Jacobson radical $\rad\La$ 
of $\Lambda$. Consider the associated exact sequence 
$$0\to S_i\to P_i \stackrel{d_3}\to  P_i^- \stackrel{d_2}\to  P_i^+ \stackrel{d_1}\to P_i \to S_i\to 0
\leqno{(*)}$$ 
where $P_i^+=P_j\oplus P_k$ and $P_i^-=P_x\oplus P_y$. We may assume that $d_1(x, y) : = \alpha x + \ba y$, since the 
induced epimorphism $(\alpha \ \ba):P_j\oplus P_k \to \Omega(S_i)=\alpha\La+\ba\La$ is a projective cover of $\Omega(S_i)$ in 
$\mod\La$. Adjusting arrows $\gamma$ or $\gamma^*$ (including impact on presentation, i.e. on generators of $I$), we can already 
say that $d_3(e_i) = (\gamma, \gamma^*)$ for some choice of the arrows $\gamma, \gamma^*$ ending at $i$ 
(see \cite[Proposition 4.3]{AGQT}). \smallskip

The kernel of $d_1$ is then $\Omega^2(S_i)=Im(d_2)$, and it has at most two minimal generators. They are images of idempotents 
$e_x\in P_x=e_x\La$ and $e_y\in P_y$ via $d_2:P_i^-\to P_i^+$. We may write them as $\vf$ and $\psi$, respectively, and they are contained 
in $P_j\oplus P_k$, so we can also write 
$$\vf = d_2(e_x,0) = (\vf_{jx}, \ \vf_{kx}) \ \ \mbox{and} \ \ \psi = d_2(0,e_y) = (\psi_{jy}, \ \psi_{ky}),$$
where $\vf_{jx}$ belongs to $e_j\La e_x$ and similarly for the other components of $\vf, \psi$. \smallskip 

The exact sequence gives information on minimal generators of the ideal $I$, which we sometimes refer to as minimal relations.
In the sense of the following lemma, arrows of $Q$ induce minimal relations.   

\begin{lemma}\label{lem:2.3} If there is an  arrow $x\to i$ then there is a minimal	generator $\rho \in e_i\La e_x$ for the ideal 
$I$ (given the presentation). \end{lemma} 

\begin{proof} Consider the generators $\vf, \psi$ of  the kernel of $d_1$. We have $\alpha \vf_{jx} + \ba \vf_{kx}=0$ in $\La$, 
equivalently the element $\alpha \vf_{jx} + \ba\vf_{kx} \in KQ$ belongs to $I$. It is a minimal relation since $\vf$ is a minimal 
generator. Similarly the generator $\psi$ gives a minimal relation in $e_i\La e_y$. \end{proof} \medskip 

Recall that any homomorphism $d:P_x\oplus P_y\to P_j\oplus P_k$ in $\mod\Lambda$ can be represented in the matrix form 
$$M={m_{jx} \ m_{jy}\choose m_{kx} \ m_{ky}},$$ 
where $m_{ab}$ is a homomorphism $P_b\to P_a$ in $\mod\La$, identified with an element $m_{ab}\in e_a\La e_b$, for any $a\in\{j,k\}$ 
and $b\in\{x,y\}$. In this way, $d$ becomes multiplication by $M$, i.e. $d(u)=M\cdot u$, for $u\in P_i^-$ (using column notation 
for vectors in $P_i^-$ and $P_i^+$). \smallskip

Continuing with the generators of $\Omega^2(S_i)$, let $M_i$ be the matrix with column the components of $\vf$ and $\psi$, 
that is $d_2$ is given by matrix 
$$M_i={\vf_{jx} \ \psi_{jy} \choose \vf_{kx} \ \psi_{ky}}.$$ 
Rewriting compositions $d_1d_2=0$ and $d_2d_3=0$ in matrix form, we get identities 
$$(\alpha \ \ba)\cdot M_i = 0\mbox{ and }M_i\cdot {\gamma \choose \gamma^*} =0\leqno{(1)} $$
for some choice of arrows $\gamma, \gamma^*$ ending at $i$ (cf. \cite[Proposition 4.3]{AGQT}). 

\begin{rem}\label{rk}\normalfont 
Basically, identities (1) determine generators (cogenerators) of $\Omega^2(S_i)$, which are encoded in columns (rows) of 
matrix $M_i$, satisfying the following \it universal properties: 
\begin{enumerate}
\item[(i)] if $\theta={\theta_1\choose\theta 2}\in P_j\oplus P_k$ is an element $\theta\in \La e_z\setminus J^2$ 
such that $[\alpha \ \ba]\cdot\theta=0$, then $z=x$ or $y$ and there is an exact sequence isomorphic to $(*)$ with $\theta$ 
being one of the columns of $M_i$, 
\item[(ii)] if $\mu\in P_x\oplus P_y$ is an element $\mu\in e_z\La \setminus J^2$ such that 
$\mu\cdot{\gamma\choose\gamma^*}=0$, then $z=j$ or $k$ and there is an exact sequence isomorphic to $(*)$ with $\mu$ being one of 
the rows of $M_i$. \end{enumerate} \normalfont 

Indeed, for $\theta$ as in (i), by definition $\theta\in Ker(d_1)=Im(d_2)$, so $\theta$ can be written as $\theta=M_i\cdot\eta$, for 
some $\eta={\eta_1\choose\eta 2}\in P_x\oplus P_y$, $\eta\in \La e_z$. Note also that all entries 
of $M_i$ are in $J$ (equivalently, $d_2$ is in $\rad_\La$), since otherwise equality $(\alpha \ \ba)\cdot M_i=0$ implies that $\alpha$ 
or $\ba$ in $J^2$ (or $\alpha\in K\ba$), which is impossible for an arrow. But $\theta\notin J^2$, i.e. $\theta_1\notin e_jJ^2e_z$ 
or $\theta_2\notin e_kJ^2e_z$, hence we infer that $\eta\notin J$, because $\eta\in J$ would force $\theta=M_i\cdot\eta\in J^2$. As 
a result, we get that $\eta_1\notin e_x J e_z$ or $\eta_2\notin e_y J e_z$. Since for $a\neq b$ in $Q_0$, we have $e_a J e_b\simeq
\rad_\Lambda(P_b,P_a)=\Hom_\Lambda(P_b,P_a)\simeq e_a \Lambda e_b$, we conclude that $z=x$ or $y$, and in both cases $\eta_\cdot$ 
is a unit of the local algebra $e_z\Lambda e_z$. We may assume that $z=x$ (the proof in case $z=y$ is similar). In particular, then 
$\eta_1$ is a unit in $e_x \Lambda e_x$ (i.e. $\eta_1$ is a scalar multiplication of $e_x$), so we obtain the following identity 
$${\theta_1 \ \psi_{jy} \choose \theta_2 \ \psi_{ky}}=M_i\cdot {\eta_1 \ 0\choose \eta_2 \ e_y}.$$ 
Denote by $M'_i$ the matrix on the left hand side and let $N={\eta_1 \ 0\choose \eta_2 \ e_y}$. Consequently, the above identity 
$M_i'=M_i\cdot N$ translates into the following commutative diagram in $\mod\La$: 
$$\xymatrix{
0\ar[r]& S_i\ar[r]& P_i\ar[r]^{d_3}& P_x\oplus P_y\ar[r]^{d_2}& P_j\oplus P_k\ar[r]^{d_1}& P_i\ar[r]& S_i\ar[r]& 0 \\ 
0\ar[r]& S_i\ar[r]\ar[u]^{id}& P_i\ar[r]^{d_3'}\ar[u]^{id}& P_x\oplus P_y\ar[r]^{d_2'}\ar[u]^{v}
& P_j\oplus P_k\ar[r]^{d_1'=d_1}\ar[u]^{id}& P_i\ar[r]\ar[u]^{id}& S_i\ar[r]\ar[u]^{id}& 0  }$$ 
where we identify $d_1=(\alpha \ \ba)$, $d_2=M_i$, $d_3={\gamma\choose\gamma^*}$, $d_2'=M_i'$, $v=N$ (which is an isomorphism, since 
$\eta_1$ and $e_y$ are units in the corresponding local algebras), and $d_3'=v^{-1}d_3$. It follows that the bottom row of the 
above diagram is the required exact sequence. Similarly, if $z=y$, then we can construct analogous matrix $M_i'$, but with $\theta$ 
as the second column. 

{\it In other words, one can swap the original sequence for the new one, in which $M_i$ admits a fixed $\theta$ as the first (or second) 
column.}  

In a similar way, we can prove (ii), where we use cokernel of $d_3$ (instead of kernel of $d_1$); indeed by the universal 
property of cokernels one can factorize matrix ${\mu_1 \ \ \mu_2 \choose \vf_{kx} \ \psi_{ky} }$ through the cokernel of $d_3$ 
($\cong Im(d_2)$), and lift this factorization to a map $u:P_j\oplus P_k\to P_j\oplus P_k$, given by a matrix 
$N={\eta_1 \ \eta_2 \choose 0 \ \ e_k }$, such that $N\cdot M_i=M_i'$ and $\eta_1$ is a unit. This means $ud_2=d_2'$, yielding 
analogous commutative diagram with (exact) isomorphic rows. \end{rem} \medskip  

Note that the conditions (i)-(ii) mentioned above explain how minimal generators (relations) of $I$ give rise to generators of 
$\Omega^2(S_i)$, and how these two are connected via the exact sequence $(*)$, up to isomorphism (here we mean both isomorphism 
of exact sequences and isomorphisms of algebras, i.e. changing presentation of $\La$). 

Namely, we may  start with a minimal generator $\rho$ of the ideal $I$ of $KQ$, without loss of generality $\rho\in e_i\La e_j$, 
where $i, j$ are vertices of $Q$. Say $\alpha, \ba$ start at $i$ and $\beta, \beta^*$ end at vertex $j$. Then we can write $\rho$ 
as an element of $KQ$ in the following way 
$$\rho = \alpha x_1\beta + \alpha x_2\beta^* + \ba x_3\beta + \ba x_4\beta^*\leqno{(2)}$$
where the $x_i$ are linear combinations of monomials, and the expression is unique if written  in terms of the monomial basis of $KQ$. 
Consequently, we infer from (i) that an element  
$$\theta = (x_1\beta + x_2\beta^*, \ x_3\beta + x_4\beta^*)$$ 
is in the kernel of $d_1$, and it can be taken as a generator for $\Omega^2(S_i)$ (column of $M_i$), for example if 
$\theta\notin J^2$. Similarly $\mu=(\alpha x_1+\ba x_3,\alpha x_2+\ba x_4)$ gives a cogenerator (row of $M_i$), if $\mu\notin J^2$. 

\begin{rem} We note that not all minimal relations can be realized in this way. Indeed, if $\La$ is a weighted surface algebra 
$\La=\La(Q,f,m_\bullet,c_\bullet)$ with at least one arrow $\alpha\in Q_1$ such that $m_\alpha n_\alpha=2$ (virtual arrow), then there 
exist a minimal zero relation of the form $\alpha\beta\gamma=0$, which cannot be induced from an element in the second syzygy of a 
simple module. \end{rem} 

\section{Triangles and squares}\label{sec:4}

In this section we discuss some properties of triangles and squares in $Q$ with respect to minimal relations. As we will see in 
Proposition \ref{prop:2.4} below, it is natural to investigate triangles in the quiver $Q$, which appear together with paths of 
length $2$ involved in minimal relations (note that this  was an essential tool in \cite[see Proposition 4.2]{AGQT}). Similarly, 
squares come with paths of length $3$ as shown in paralell result (Lemma \ref{lem:2.8}). 

\medskip 

If $p$ is a monomial in $KQ$, we write $p\prec I$, provided that $p$ occurs as a term (summand) in some minimal relation defining $I$ 
(i.e. $p$ is involved in a minimal relation). Very often, paths of length two occur in this way as shown in \cite[Proposition 4.2]{AGQT}.   

\subsection{Paths of length 2 and triangles}\label{ssec:3.1}

\begin{prop}\label{prop:2.4} Assume $\alpha: i\to j$ and $\beta: j\to k$ are arrows such that $\alpha\beta \prec I$. Then there
is an arrow in $Q$ from $k$ to $i$, so that  $\alpha$ and $\beta$ are part of a triangle in $Q$. \end{prop} 

\begin{proof} Recall we write $\bar{\alpha}$ for the other arrow starting at $i$ (if it exists), and write $\beta^*$ for the other 
arrow ending at $k$ (if it exists). Then $\alpha\beta \prec I$ means that 
$$\alpha\beta + \alpha z_0\beta + \alpha z_1\beta^* + \ba z_2\beta + \ba z_3\beta^* = 0$$ 
in $\La$ where $z_0\in J$, and $z_i\in \La$. We may assume $z_0=0$, otherwise we replace $\alpha$ by $\alpha(1+z_0)$. We use the exact 
sequence $(*)$. Then the identity above gives an element $\vf$ in the kernel of $d_1$, namely
$$\vf = (\beta + z_1\beta^*, z_2\beta + z_3\beta^*)$$
Clearly, $\vf\notin J^2$, because its first coordinate admits an arrow. Therefore, using Remark \ref{rk}(i) for $\theta:=\vf$ (viewed 
as a column), we conclude that $P_k$ is a direct summand of $P^-_i$, i.e. $k$ is a source of an arrow ending at $i$, and the claim 
follows. \end{proof} 

Note that this holds for any symmetric periodic algebra of period 4 (i.e. also for wild ones). 

\medskip

\begin{expl} \label{expl:2.6}\normalfont In  \cite[Section 11]{AGQT}, there is a quiver $Q$ of an algebra which is symmetric and 
periodic of period 4, but the algebra is wild (so out of our current interest; see also \cite{BIKR} and \cite[Corollary 2]{AGQT}). 
This is mentioned as a consequence of the classification in \cite{AGQT}, however we can observe it already, as the above 
proposition implies that the algebra must be wild. \smallskip

Namely, it follows from Proposition \ref{prop:2.4} that any path $\rho$ in $Q$ of length two which does not involve a loop satisfies 
$\rho\nprec I$. Therefore $B/J^3$ contains a wild subalgebra, given by a quiver of type $\widetilde{\widetilde{E}}_7$ without any 
relations, as in \cite[see the proof of Proposition 4.2]{AGQT}. \end{expl} 

\begin{lemma}[Triangle Lemma]\label{lem:2.7}  Assume $Q$ contains  a triangle
        \[
 \xymatrix@R=3.pc@C=1.8pc{
%  \xymatrix@C=.8pc{
    x
    \ar[rr]^{\gamma}
    && i
    \ar@<.35ex>[ld]^{\alpha}
    \\
    & j
    \ar@<.35ex>[lu]^{\beta}
  }
\]
If $\gamma\alpha \not\prec I$ then also $\alpha\beta\not\prec I$. \end{lemma}

\begin{proof} Consider the exact sequence for the simple module $S_x$ 
$$0\to S_x \to P_x \to P_j\oplus P_{j^*}   \to P_i\oplus P_{\bar{i}} \to  P_x\to S_x \to 0,$$ 
where $j^*=s(\beta^*)$ and $\bar{i}=t(\bar{\gamma})$. Taking minimal generators for $\Omega^2(S_x)$ gives the columns of the matrix 
$M_x$, that is 
$$M_x = \left(\begin{matrix}\vf_{ij} & \psi_{ij^*}\cr \vf_{\bar{i}j} & \psi_{\bar{i}j^*}\end{matrix}\right).$$
It satisfies $(\gamma \ \bar{\gamma})\cdot M_x=0$ and $M_x\cdot{\beta \choose \beta^*}=0$.

Suppose $\gamma\alpha\nprec I$, but (for a contradiction) $\alpha\beta\prec I$. Then there is a minimal relation of the form
$$\alpha\beta + \alpha z_1\beta + \alpha z_2\beta^* + \ba z_3\beta + \ba z_4\beta^*=0 $$
with $z_1\in J$, and we may assume again $z_1=0$. Now, if we define
$$\theta:= (\alpha + \ba z_3, \alpha z_2 + \ba z_4),$$
then $\theta\cdot {\beta\choose \beta^*}=0$ and $\theta\notin J^2$ (since it involves an arrow), so by Remark\ref{rk}(ii), we can 
take $\theta$ as the first row of $M_x$. In particular, $\vf_{ij}=\alpha+\ba z_3$, hence it follows that 
$\gamma(\alpha + \ba z_3) + \bar{\gamma} \vf_{\bar{i}j} =0$, and we obtain $\gamma\alpha \prec I$, a contradiction. 
\end{proof} \medskip 

Consequently, for a triangle as above, either all paths $\gamma\alpha,\alpha\beta,\beta\gamma\prec I$, or none of them. Hence 
triangles in $Q$ split into two families: triangles, say \emph{of type R} (for which all $\gamma\alpha,\alpha\beta,\beta\gamma\prec I$) 
and triangles \emph{of type N}. We will further see (Section \ref{sec:5}) similar distinction between non-regular vertices. \smallskip 

Let us finish this part with the following lemma. 

\begin{lemma}\label{lem:2.9} Assume $i$ is a 1-vertex which is part of a triangle 

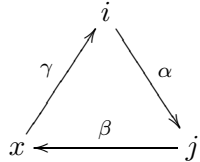
\[
%  \xymatrix@R=2pc@C=1.5pc{
%  \xymatrix@R=3.5pc@C=1.8pc{
  \xymatrix@R=3.pc@C=1.8pc{
%  \xymatrix@C=.8pc{
    & i
    \ar[rd]^{\alpha}
    \\
  x 
    \ar[ru]^{\gamma}
  %  \ar@<-.5ex>[rr]_{\eta}
    && j
   \ar[ll]_{\beta}
% \ar[ld]^{\beta}
%    \\
%    & d
 %   \ar[lu]^{\nu}
  }
\]

Then both $x$ and $j$ must be 2-vertices. \end{lemma}

\begin{proof} By Lemma \ref{lem:2.1}, there must be another arrow, say $\bar{\gamma}$, starting at $x$, and there must be another 
arrow, say $\alpha^*$, ending at $j$. From the exact sequence for $S_i$ we know that $p_x = p_j$.

(i) Assume vertex $x$ is not a $2$-vertex, then $\beta$ is the only arrow ending at $x$. Therefore 
$$e_x\La/S_x \cong \beta\La.$$
Moreover, again by Lemma \ref{lem:2.1}  there must be another arrow starting at $j$, call it $\bar{\beta}$. Hence 
$${\rm rad} (P_j) = \beta\La + \bar{\beta}\La.$$
As a result, we get the following equalities of dimension vectors: 
$$p_x=s_x+\ul{\dim}(\beta\Lambda)\mbox{ and }p_j=s_j+\ul{\dim}(\beta\La+\bar{\beta}\La)=s_j+ \ul{\dim}\beta\La + 
\ul{\dim}(\bar{\beta}\La/\beta\La \cap \bar{\beta}\La).$$ 
Comparing $p_x=p_j$, we conclude that $\ul{\dim}(\bar{\beta}\La/\beta\La \cap \bar{\beta}\La)=s_x-s_j$, so this must be zero 
(i.e. $x=j$), since otherwise $s_x-s_j$ has a negative coordinate, which cannot happen for a dimension vector of a $\La$-module. 

In particular since $S_j, S_x$ are simple and the vector space dimension of $\beta\La\cap \bar{\beta}\La$ is equal to the vector 
space dimension of $\bar{\beta}\La$. However, we have an inclusion of these spaces, so they are equal. Now 
$\bar{\beta}\La = \beta\La \cap \bar{\beta}\La  \subseteq \beta\La$, hence $\bar{\beta}\in \beta\La$. But this is not possible 
since $\bar{\beta}$ is an arrow $\neq \beta$.  

(ii) The proof that $j$ must be a 2-vertex is dual. \end{proof} 

\subsection{Paths of length $\geq 3$}\label{ssec:3.2} In this short paragraph we will consider a bit longer paths, i.e. of length 
$3$ or $4$ (and in particular, induced squares). Suppose the quiver of $\La$ has a subquiver 
$$ u\stackrel{\delta}\longrightarrow i  \stackrel{\alpha}\longrightarrow k \stackrel{\beta}\longrightarrow t 
\stackrel{\gamma}\longrightarrow j $$ 
We have the following counterpart of Proposition \ref{prop:2.4}. 

\begin{lemma}\label{lem:2.8} Suppose $\alpha\beta \not\prec I$ or  $\beta\gamma\not\prec I$, and $\alpha\beta\gamma \prec I$. 
Then there is an arrow $j\to i$. \end{lemma} 

\begin{proof} We work in $KQ$, using the basis consisting of paths. We deal with the case $\alpha\beta \not\prec I$ (the other case 
is dual, working with the opposite algebra).

After possibly adjusting arrow $\beta$, there is a minimal relation of the form
$$\alpha\beta\gamma + \alpha z_1\gamma^* + \ba z_2\gamma + \ba z_3 \gamma^* \in I$$
Consider the exact sequence for the simple module $S_i$,
$$0\to S_i \to P_i \to P_u\oplus P_{u'} \to P_k\oplus P_l\to P_i\to S_i\to 0 $$
Here $\alpha:i\to k, \ba:i\to l$ start at $i$, and $\delta: u\to i$ and $\delta^*: u'\to i$ are the arrows ending at $i$, where by 
convention, $\ba$ or $\delta^*$ may not exist (then we omit $P_l$ and $P_{u'}$). 

We take $\Omega(S_i) = \alpha \La + \ba \La$ and $\Omega^2(S_i) =  \{ (x, y)\in P_k\oplus P_l\mid \alpha x + \ba y = 0\}$. From the 
exact sequence, this is equal to $\vf \La + \psi \La$ where $\vf = \vf e_u$ and $\psi = \psi e_{u'}$.  Let $M_i$ be the matrix with 
columns $\vf$ and $\psi$. Then (for some choice of arrows $\delta, \delta^*$) we have
$$(\alpha \ \ba)\cdot M_i=0 \ \  M_i\cdot {\delta \choose \delta^*}=0. $$

The minimal relation above gives rise to the following element $\theta$ which belongs to $\Omega^2(S_i)$,
$$\theta = (\beta\gamma + z_1\gamma^*, \ z_2\gamma + z_3\gamma^*)$$
If $\theta\notin J^2$, then we may take $\vf:=\theta$ as the first column of $M_i$, since $\beta\gamma\prec \theta_1$ must be 
in $e_k\La e_u$ (see also Remark \ref{rk}(i)). It follows that $j=u$, and hence $\delta$ is an arrow from $j=u$ to $i$.

\bigskip

Suppose now $\theta\in J^2$. We will show that this leads to a contradiction. The radical of $\Omega^2(S_i)$ is equal to 
$\Omega^2(S_i)J=\vf J + \psi J$. So we can write $\theta = \vf v + \psi w$ and $v, w\in J$ and we can take them in $Je_j$. Then
$$\beta\gamma + z_1\gamma^* = \vf_{ku}v + \psi_{ku'}w$$
Say $\beta\gamma$ occurs in $\vf_{ku}v$. We can write $v= ve_j = v_1\gamma + v_2\gamma^*$ with $v_i\in \La$ (which need not be in the 
radical). We can write $\vf_{ku} = \beta y_1 + \bar{\beta} y_2$ with $y_1, y_2\in KQ$. Then 
$$\vf_{ku}v = \beta y_1v_1\gamma + \beta y_1v_2\gamma^* + \bar{\beta}y_2v $$
Then $\beta\gamma$ is a term of $\beta y_1v_1\gamma$. Therefore $y_1v_1$ (which we can take equal to  $y_1v_1e_t$) is equal to $e_t$ 
modulo the radical. It follows that $y_1v_1$ is a unit in $e_t\La e_t$. 

However, it factors through vertex $u$, and {\it it follows that $u=t$}. So we have a triangle $(\alpha, \beta, \delta)$ and 
$\alpha\beta \not\prec I$. It follows (from our triangle lemma) that also $\beta\delta \not\prec I$. 

\bigskip

On the other hand, we exploit the identity for $\vf_{ku}v$ a bit further. Since $y_1v_1$ is a unit, we may assume $\beta = \beta y_1$. 
Then $\vf_{ku} = \beta + \bar{\beta}y_2$ and recall this is the top left entry of the Matrix $M_i$ above. We have 
$M_i{\delta\choose \delta^*} = 0$ which gives
$$\beta\delta + \bar{\beta}y_2\delta + \psi_{ku'}\delta^*=0,$$
hence $\beta\delta \prec I$, a contradiction. \end{proof} 

\bigskip 
The above proposition shows that paths of length $3$ involved in minimal relations induce squares in $Q$. We have a result similar 
to previous Triangle Lemma, stated as follows. 

\begin{lemma}[Square Lemma] Assume $Q$ contains a square 
$$\xymatrix{1\ar[r]^{\delta} & 2\ar[d]^{\alpha} \\ 4\ar[u]^{\gamma} & 3\ar[l]_{\beta}}$$ 
If $\alpha\beta\gamma\prec I$, then $\beta\gamma\delta\prec I$. \end{lemma}  

\begin{proof} Suppose that $\alpha\beta\gamma\prec I$, but $\beta\gamma\delta\nprec I$. In particular, we have also 
$\beta\gamma \nprec I$. Consider the exact sequence for $S_2$. Then $\Omega^2(S_2)$ has generators being the columns of the 
matrix 
$$M_2 = \left(\begin{matrix} \vf_{31} & \psi_{3, 1^*}\cr \vf_{\bar{3}1} & \psi_{\bar{3}1^*}\end{matrix}\right).$$
and $(\alpha \ \ba)M_2=0$ and $M_2{\delta\choose \delta^*}=0$. By our assumption  we have minimal relation
$\alpha(\beta\gamma + x_1\gamma +  x_2\bar{\gamma}) + \ba(x_3\gamma + x_4\bar{\gamma}) = 0$ with $x_1\in J^2$. This gives an 
element 
$$\vf = (\beta\gamma + x_1\gamma + x_2\bar{\gamma}, \ x_3\gamma + x_4\bar{\gamma})^t\in\Omega^2(S_2)$$
which cannot be in the radical of $\Omega^2(S_2)$, since $\beta\gamma \not\prec I$. So we can take this as the first column of $M_2$. 
It follows that $(\beta\gamma + x_1\gamma + x_2\bar{\gamma})\delta + \psi_{31^*}\delta^*=0$, so $\beta\gamma\delta \prec I$, and we 
get a contradiction. \end{proof} \medskip 

The following lemma shows that sometimes one can relate paths of length 3 and 4. 

\begin{lemma}\label{lem:4.7} Suppose we have a path in $Q$ of the form 
$$ u\stackrel{\delta}\longrightarrow i  \stackrel{\alpha}\longrightarrow k \stackrel{\beta}\longrightarrow t 
\stackrel{\gamma}\longrightarrow j $$  
with $|k^+|=1$ and $\delta\alpha,\alpha\beta,\beta\gamma\nprec I$. If $\delta\alpha\beta\gamma\prec I$, then 
$\alpha\beta\gamma\prec I$. \end{lemma} 

\begin{proof} Suppose that $\delta\alpha\beta\gamma\prec I$ and let $\bar{\delta}:u\to u'$ and $\bar{\alpha}:i\to i'$ denote the second 
arrow starting at $u$ and $i$ (if exist). By the assumptions on $k$, we conclude that any path $p\in e_i A$ starting from $\delta\alpha$ 
must go through $\delta\alpha\beta$. Hence $\delta\alpha\beta\gamma$ as a minimal generator of $I$ is involved in a minimal relation 
of the form 
$$\delta\alpha\beta\gamma+ \delta\alpha\beta p+\delta\bar{\alpha}q+\bar{\delta}r=0,$$  
where $p,q,r\in J$. After adjusting $\gamma:=\gamma+p$, we may change the presentation to get $p=0$. Consequently, the element 
$\rho=(\alpha\beta\gamma+\bar{\alpha}q,r)$ belongs to $\Omega^2(S_u)=ker([\delta \ \bar{\delta}])$. Finally, if 
$u^-=\{\sigma,\sigma^*\}$, and $v=s(\sigma)$, $v'=s(\sigma^*)$, then $\Omega^{2}(S_u)\cong Im(M_u)$, where 
$M_u:P_v\oplus P_{v'}\to P_i\oplus P_{i'}$ is given by the matrix 
$$\left(\begin{matrix} \vf_{uv} & \psi_{u v'}\cr \vf_{u'v} & \psi_{u'v'}\end{matrix}\right),$$ 
hence $\rho=M_u\cdot {\kappa_1 \choose \kappa_2}$, for some $\kappa_1\in P_{v}$ and $\kappa_2\in P_{v'}$. But then 
$$\alpha\beta\gamma+\bar{\alpha}q=\vf_{uv}\kappa_1+\psi_{uv'}\kappa_2,$$ 
and therefore, $\alpha\beta\gamma$ is generated by minimal relations. But, $\alpha\beta\nprec I$ and $\beta\gamma\nprec I$, hence 
$\alpha\beta\gamma$ is also involved in some minimal relation of $I$, as claimed. \end{proof}

\bigskip
 
\section{Non-regular vertices}\label{sec:5} 

In this section, we give some partial results describing non-regular vertices. \smallskip 

Clearly, for $Q$ biserial, the non-negular vertices $i$ satisfy either $|i^-|=1$ and $|i^+|=2$ or $|i^-|=2$ and $|i^+|=1$. In the 
first case $i$ is called a $(1,2)$-vertex, whilst in the second a $(2,1)$-vertex. Let $i$ be a $(1,2)$-vertex of the form 
$$\xymatrix@R=0.4cm{j\ar[r]^{\alpha} & i \ar[r]^{\beta}\ar[rd]_{\bar{\beta}} & k \\ && l}$$ 
We call $i$ a vertex \emph{of type R} (respectively, \emph{of type N}), provided that both $\alpha\beta\prec I$ and 
$\alpha\bar{\beta}\prec I$ (respectively, both $\alpha\beta\nprec I$ and $\alpha\bar{\beta}\nprec I$). If $k\neq l$, $i$ is said 
to be \emph{proper}. Similar notions can be defined for $(2,1)$-vertices. Whenever we consider a $(1,2)$-vertex $i$, we keep the 
above notation for arrows starting and ending at $i$. 

\begin{rem} \normalfont We recall that there exist infinitely many pairwise non-isomorphic TSP4 algebras $A$ containing arbitrary large 
number of $(1,2)$- and $(2,1)$-vertices of both types R or N. Indeed, one may take any weighted surface algebra $\La$ (see \cite{WSA-GV}) 
containing arbitrary number of 'blocks' of the form 
$$
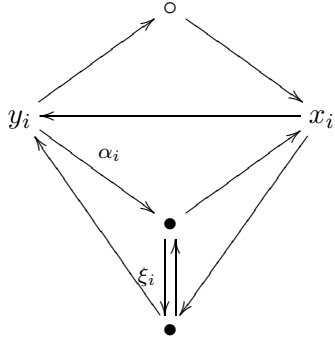
\xymatrix@R=1cm@C=1.5cm{& \circ \ar[rd]& \\
y_i\ar[rd]^{\alpha_i}\ar[ru]& &  x_i\ar[ll]\ar[ldd] \\ 
&\bullet\ar[ru]\ar@<-0.07cm>[d]_{\xi_i}& \\ 
&\bullet\ar@<-0.07cm>[u]\ar[luu]&}$$  
with $\xi_i$ being a virtual arrow. Then using results of \cite[see also Section 4]{HSS}, we conclude that the virtual mutation 
$A=\La(\xi)$ with respect to the sequence $\xi=(\xi_1,\dots,\xi_n)$ of virtual arrows is a TSP4 algebra and vertices $x_1,\dots,x_n$ are 
$(1,2)$-vertices of type N (in $Q_A$), whereas $y_1,\dots,y_n$ are $(2,1)$-vertices of type N (in definition of $\La$, we have to pick 
weights $m_{\xi_i}=m_{\alpha_i}=1$, for any $i\in\{1,\dots,n\}$). \smallskip 

For vertices of type R one has to consider so called {\it weighted generalized triangulation algebras} \cite{SS}, given by quivers 
which are glueings of blocks of five types I-V. Without going into details, we only mention that for any such algebra $A$ (it is a TSP4 
algebra in most cases), its Gabriel quiver contains two $(1,2)$-vertices and two $(2,1)$-vertices per each block of type V, and all 
these vertices are of type R (this follows directly from the shape of relations in $A$). Therefore, one can easily construct a TSP4 
algebra with arbitrary large number of non-regular vertices of type R (in this case the number of $(1,2)$-vertices is equal to the 
number of $(2,1)$-vertices). But then the Gabriel quiver of $A$ is not biserial. \end{rem} 

It turns out that there are no non-regular vertices of type R, if the Gabriel quiver is biserial. For proper ones, it is pretty easy 
to see, as the following lemma shows. 

\begin{lemma}\label{poper} There are no proper non-regular vertices of type R. \end{lemma} 

\begin{proof} Suppose $i$ is a $(1,2)$-vertex of type R. In particular, $\alpha\beta\prec I$ yields an arrow $\gamma:k\to j$, whereas 
$\alpha\bar{\beta}:l\to j$, an arrow $\delta:l\to j$, due to Proposition \ref{prop:2.4}. Suppose now $i$ is proper. Then 
$j^-=\{\gamma,\delta\}$, since $Q$ is biserial, and hence $p_j^-=p_k+p_l$. But $p_i^-=p_i^+$ gives $p_j=p_k+p_l$, because $i$ is a 
$(1,2)$-vertex, and therefore, we get $p_j=p_j^-=\hat{p}_j$, which is a contradiction with Lemma \ref{lem:1}. Dual arguments provide 
the proof for $(2,1)$-vertices. \end{proof} 

We complete the claim as follows. 

\begin{thm} There are no non-regular vertices of type R. \end{thm} 

\begin{proof} By previous lemma, it is sufficient to prove that there is no $(1,2)$-vertex $i$ of type R with $k=l$ (i.e. non-proper 
one). Suppose to the contrary, that such a vertex exists. For simplicity, we will use notation $1,2,3$ for vertices, respectively, 
$i,k=l$ and $j$. Since $1=i$ is of type R, we get an arrow $\gamma:2\to 3$ (see Proposition \ref{prop:2.4}), and consequently, $Q$ 
admits the following subuiver 
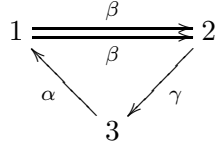
$$\xymatrix{1\ar@<+0.35ex>[rr]^{\bar{\beta}}\ar@<-0.35ex>[rr]_{\beta} & & 2\ar[ld]^{\gamma} \\ & 3\ar[lu]^{\alpha} & }$$ 
Of course, $1$ is a non-regular vertex, by the assumption. We claim that also $2$ is a non-regular vertex. Indeed, if this is not the 
case, then there is an arrow $\bar{\gamma}:2\to x$, $\bar{\gamma}\neq\gamma$, and moreover, $x\neq 3$, because otherwise, we would 
get a subquiver $\xymatrix@C=0.35cm{1\ar@<+0.35ex>[r]\ar@<-0.35ex>[r]&2\ar@<+0.35ex>[r]\ar@<-0.35ex>[r]&3}$, hence $p_1=p_3$, because 
$p_2^-=p_2^+$, and so $\hat{p}_1=p_1^-=p_3=p_1$, which gives a contradiction with Lemma \ref{lem:1}. Consequently, we have no arrows 
$x\to 1$, so both $\beta\bar{\gamma}\nprec I$ and $\bar{\beta}\bar{\gamma}\nprec I$, due to Proposition \ref{prop:2.4}. But then $A$ 
admits a wild (hereditary) factor algebra of the form $\xymatrix@C=0.35cm{1\ar@<+0.35ex>[r]\ar@<-0.35ex>[r]&2\ar[r]^{\bar{\gamma}}&x}$, 
a contradiction. This proves that $2$ is a $(2,1)$-vertex. \smallskip 

In particular, using $p_1^-=p_1^+$ and $p_2^-=p_2^+$ one gets $p_1=p_2$ and $p_3=2p_1$. It follows also that $3^-\cup 3^+\supsetneq
\{\gamma,\alpha\}$, and hence $3$ is a $2$-vertex (note: $p_3^-=p_3^+$). In particular, there may be a loop $\rho$ at vertex $3$, and 
then $Q_1=\{\alpha,\beta,\bar{\beta},\gamma,\rho\}$). Otherwise, instead of $\rho$ there are arrows starting and ending at $3$, and 
other vertices or arrows. We start with properties which hold in both cases. \smallskip

We may assume that $e_1J^3 = e_1\beta\gamma J$ and $\bar{\beta}\gamma \in e_1J^3$. By tameness of $A$, there must be a minimal relation 
involving at least one of $\beta\gamma, \bar{\beta}\gamma$. This means that $\dim e_1J^2/e_1J^3\leq 1$ and clearly it is non-zero. So 
we may assume it is spanned by the coset of $\beta\gamma$. Then $\bar{\beta}\gamma = c\beta\gamma + \psi$ for $c\in K$ and 
$\psi \in e_1J^3$. If $c\neq 0$ then we replace the arrow $\bar{\beta}$ by $\bar{\beta} - c\beta$, to get the claim. 

We write the relation as
$$\bar{\beta}\gamma = \beta\gamma \Theta_1\gamma + \beta\gamma\Theta_2\rho\ \ 
 \ (\Theta_1\in J^2, \ \  \Theta_2\in \La).  \leqno{(*)} $$

Note also that we may assume $\gamma\alpha = 0$. Indeed, using the exact sequence for the simple module $S_1$, one can see that the 
second syzygy $\Omega^2(S_1)$ has one minimal generator, say $(\vf_{23}, \psi_{23})$, satisfying both $\vf_{23}\alpha=0$ and 
$\psi_{23}\alpha =0$. Applying now the relation (*), we may fix the generator of the form $(-(\gamma\Theta_1\gamma + \gamma\Theta_2\rho),\gamma)$ (previously adjusting $\gamma$, to get $\psi_{23}=\gamma$). In particular, $\gamma\alpha =0$ for this choice. 
\smallskip

Consider the exact sequence for $S_2$. This gives
$$0\to \Omega^{-1}(S_2) \cong (\beta, \bar{\beta})\La \to P_1\oplus P_1 \to P_3 \to \gamma \La \cong \Omega(S_2)\to 0 $$ 
Since $\gamma\alpha=0$ we have $\alpha\La \subset \Omega^2(S_2)$. It is not in the radical of $\Omega^2(S_2)$ since 
$\alpha$ is an arrow. From the exact sequence, $\Omega^2(S_2)$ has one more generator, call it $\vf_{31}$, and 
we have the minimal relation
$$\alpha\beta + \vf_{31}\bar{\beta} = 0 \leqno{(**)}$$

Assume now that there is a loop at vertex $3$ and consider the exact sequence for $S_3$. This gives an exact sequence 
$$0\to \Omega^{-1}(S_3) \cong (\gamma, \rho')\La \to P_2\oplus P_3 \to \ P_1\oplus P_3 \to \alpha\La \oplus \rho\La  
\cong \Omega(S_3)\to 0 $$
where $\rho'$ is a version of $\rho$, and the middle map is given by matrix $M_3={\vf_{12} \ \psi_{13} \choose \vf_{32} \ \psi_{33}}$, 
which describes the generators of $\Omega^2(S_3)$. Then $(\alpha \ \rho)M_3=0$ and $M_3{\gamma \choose \rho'} = 0$. We can write the 
minimal relation (**) as $0=\alpha\beta \ + \ \alpha\phi'\bar{\beta} + \rho\phi''\bar{\beta}$ where $\phi'\in \La$ and $\phi''\in J$, 
and therefore $\Omega^2(S_3)$ has a generator $(\beta + \phi'\bar{\beta},\phi''\bar{\beta})$. This can be taken as the first column 
in  $M_3$. The first row of $M_3$ gives now that we have a minimal relation $(\beta + \vf'\bar{\beta})\gamma + \psi_{13}\rho'=0$. Now 
$\bar{\beta}\gamma$ is in $\beta\gamma J$ and $\psi_{13}\rho'$ also is in $J^3$. Therefore $\beta\gamma\in J^3$. But we have seen 
that $e_1J^3 = \beta\gamma J$, so we deduce $e_1J^3  \subseteq e_1J^4$ and hence $e_1J^3=0$. As a result, $\beta\gamma J=0$ and 
$\beta\gamma\in {\rm soc}(e_1\La) \cap e_1\La e_3=0$, which is not possible for a symmetric algebra (see \cite[I.3.5]{E1}). \smallskip 
 
In general, we have arrows $\ba: 3\to x$ and $\gamma^*: y\to 3$. The exaxt sequence for $S_3$ is now of the form 
$$0\to \Omega^{-1}(S_3) \cong (\gamma, \gamma^*)\La \to P_2\oplus P_y \to \ P_1\oplus P_x \to \alpha\La \oplus \ba\La  
\cong \Omega(S_3)\to 0 $$
Exactly as in the first case, we rewrite (**) which gives the first column of the matrix $M_3$. Then from the first row we get the 
identity $0 = (\beta + \phi'\bar{\beta})\gamma + \psi_{1y}\gamma^*$, and we get the same contradiction as before. \end{proof}

\end{document}